\documentclass[12pt,reqno]{amsart}
\usepackage{cite}

\usepackage{verbatim,amscd,amssymb}
\usepackage{graphicx}
\usepackage{amsmath}

\usepackage{inputenc}

\usepackage{enumitem}
\usepackage{float}

\usepackage[active]{srcltx}

\graphicspath{ {./images/} }

\usepackage{fontenc}

\usepackage{float}

\usepackage[color, notref, notcite]{showkeys}
\definecolor{refkey}{gray}{.5}
\definecolor{labelkey}{gray}{.5}

\newtheorem{theorem}{Theorem}[section]
\newtheorem{lemma}[theorem]{Lemma}
\newtheorem{cor}[theorem]{Corollary}

\theoremstyle{definition}
\newtheorem{definition}[theorem]{Definition}

\theoremstyle{remark}

\numberwithin{equation}{section}

\newcommand{\R}{\mathbb R}
\newcommand{\uc}{\mathbb{S}}

\newcommand{\orp}{\mathrm{orp}}

\newcommand{\sha}{\succ\mkern-14mu_s\;}

\newcommand{\oa}{\overrightarrow}

\begin{document}
	
	\date{\today}
	
	\title[Monotonicity of the over-rotation intervals]{Monotonicity of the over-rotation intervals for bimodal maps}

	\author{Sourav Bhattacharya}
	
	\author{Alexander Blokh}
	
	\address[Sourav Bhattacharaya and Alexander Blokh]
	{Department of Mathematics\\ University of Alabama at Birmingham\\
		Birmingham, AL 35294}
	\email{Sourav Bhattacharya :sourav@uab.edu}
	\email{Alexander Blokh :ablokh@math.uab.edu}
	
	
	
	\begin{abstract}
		We show that the connectedness of the set of parameters for which the \emph{over-rotation interval}
		of a bimodal interval map is constant. In other words, the over-rotation interval is a monotone function of a bimodal interval map. 	
	\end{abstract}
	
	\maketitle
	
	\section{INTRODUCTION}
	
	One-dimensional combinatorial dynamics started when O. M. Sharkovsky proved his theorem on
	the coexistence of periods of cycles for interval maps. To state it, we recall the (transitive)
	\emph{Sharkovsky order} of the set natural numbers:
	\begin{multline*}
		3\sha 5\sha 7\sha\dots\sha 2\cdot3\sha 2\cdot5\sha 2\cdot7\sha\\ \dots
		\sha 2^2\cdot3\sha 2^2\cdot5\sha 2^2\cdot7\sha\dots\sha 2^2\sha 2\sha 1.
	\end{multline*}
	In what follows by the \emph{period} we mean the \emph{minimal} period.
	Below $I$ always denotes a closed interval.
	
	\begin{theorem}[\cite{S} and~\cite{shatr} for English translation]\label{t:shar}
		If $g:I\to I$ is continuous, $m\sha n$ and $m$ is the period of cycle of $g$
		then $n$ is also the period of a cycle of $g$.
	\end{theorem}
	
	Theorem \ref{t:shar} inspired a lot of developments. One of them is the discovery of
	\emph{over-rotation numbers}, a combinatorial tool used for classifying interval cycles and
	invariant under topological conjugacy of interval maps. It can be defined as follows.
	
	Let $I$ be the unit interval, let $f:I\to I$ be a continuous interval map, let
	$P$ be a periodic orbit of $f$ of period $q>1$, and let $m$ be the number of
	points $x \in P$ such that $f(x) -x$ and $f^2(x)-x$ have different signs.
	Call the pair $(\frac{m}{2},q)=\orp(P)$ the \emph{over-rotation pair} of $P$ and
	denote it by $orp(P)$; call $\frac{m}{2q}=\rho(P)$ the \emph{over-rotation number}
	of $P$. An over-rotation pair $(p,q)$ is \emph{coprime} if $p$ and
	$q$ are coprime.
	The set of \emph{over-rotation pairs} of all cycles of
	$f$ is denoted by $ORP(f)$. Since the number $0<m\le \frac{q}{2}$
	is even, then, in an \emph{over-rotation pair} $(p,q)$, $p$ and $q$ are
	integers and $0<\frac{p}{q}\le \frac{1}{2}$. The number $p$ can
	be interpreted as the number of times $f(x)$ goes around $x$ as we move
	along the orbit of the ``vector'' $\oa{x f(x)}$.
	If $f$ has a
	unique fixed point $a$, then $f(x)>x$ if $x<a$ and $f(x)<x $ if
	$x>a$, and $p$ is the number of points to the right of $a$ which are
	mapped to the left of $a$ (or vice versa).
	
	\begin{theorem}[\cite{BM1}]\label{overrotation:pairs}
		Suppose that $(p,q)$ and $(r,s)$ are \emph{over-rotation pairs}. Moreover, suppose that one
		of the following holds.
		
		\begin{enumerate}
			
			\item  $\frac{p}{q}< \frac{r}{s}$.
			
			\item $\frac{p}{q}=\frac{r}{s}$ so that for a coprime over-rotation pair $(k, l)$
			we have $p/k=q/l=u$ and $r/k=s/l=v$ are integers, and $u\sha v$.
			
		\end{enumerate}
		
		Then any interval map with a cycle of over-rotation
		pair $(p,q)$ has a cycle of over-rotation pair $(r,s)$.
	\end{theorem}
	
	Definition \ref{d:orint} is based upon Theorem \ref{overrotation:pairs}; to rule out trivial cases we will
	from now on assume that we consider only maps with non-fixed periodic points
	(otherwise orbits of all points converge to a fixed point).
	
	\begin{definition}\label{d:orint}
		Given an interval map $f$, denote by $I_f$ the closure of the union of over-rotation
		numbers of $f$-periodic points, and call $I_f$ the \emph{over-rotation interval} of $f$.
	\end{definition}
	
	By Theorem \ref{t:shar}, any map $f$ with non-fixed periodic
	points has a cycle of period 2 with over-rotation number $\frac{1}{2}$;
	by Theorem \ref{overrotation:pairs} if $\frac{p}{q}=\rho(P)$ for a cycle $P$ then
	$[\frac{p}{q}, \frac12]\subset I_f$; hence for any interval map $f$ there exists
	a number $r_f, 0\le r_f<\frac12,$ such that $I_f=[r_f, \frac12]$.
	
	A continuous map $f:[0,1] \to [0,1]$ is said to have a \emph{horseshoe}
	if there exist subintervals $I$ and $J$ of $[0, 1]$ disjoint except
	perhaps a common endpoint, such that $f(I)\cap f(J)\supset I\cup J$. By
	\cite{BM1}, if $f$ has a horseshoe then $I_f = [0, \frac{1}{2}]$.
	
	A map is \emph{piecewise-monotone} if the interval $I$ can be
	partitioned into finitely many intervals (\emph{laps}) on which $f$ is
	\emph{strictly monotone}. If the smallest number of such intervals is $2$, then
	$f$ is said to be \emph{unimodal}; in this case $f$ has one turning point.
	If the smallest number of such intervals is $3$, then
	$f$ is said to be \emph{bimodal}; in this case $f$ has two turning point.
	In this paper we \emph{always} consider \emph{only} piecewise-monotone maps
	that are increasing on their leftmost laps.

	Every cycle $P$ of a map $f$ induces a cyclic permutation $\Pi$
	obtained by looking at how the map $f$ acts on the points of $P$
	ordered from the left to the right. One can introduce a relation $\sim$
	on the family of all cycles such that for two cycles $P$ and $Q$,
	$P\sim Q$ iff $P$ and $Q$ induce the same permutation; $\sim$ is an
	equivalence relation whose equivalence classes are called
	\emph{patterns}. If an interval map $f$ has a cycle $P$ from a
	\emph{pattern} $\pi$ associated with permutation $\Pi$, say that $P$ is
	a \emph{representative} of $\pi$ (in $f$) and $f$ \emph{exhibits} $\pi$
	(on $P$). A permutation $\pi$ forces a permutation $\theta$ if any
	continuous interval map $f$ which exhibits $\pi$ also exhibits
	$\theta$. By \cite{Ba}, \emph{forcing} is a partial ordering.
	
	A useful algorithm allows one to describe all patterns forced by a
	pattern $\pi$. Consider a cycle $P$ of pattern $\pi$, and denote the
	leftmost point of $P$ by $a$ and the rightmost point of $P$ by $b$.
	Every component of $[a,b]-P$ is called a \emph{P-basic interval}.
	Extend this map from $P$ to $[a,b]$ by defining it linearly on each
	$P$-basic interval and call the resultant map $f_P$ the
	$P$-\emph{linear map}. Then, the patterns of all cycles of $f_P$ are
	exactly the patterns forced by the pattern of $P$ (see \cite{alm00} and
	\cite{Ba}).
	
	Over-rotation pairs and numbers for patterns are defined just like for
	an interval map. Denote the \emph{over-rotation pair} and
	\emph{over-rotation number} of a pattern $\pi$ by $\orp(\pi)$ and
	$\rho(\pi)$, resp. If $P$ is a periodic orbit of pattern $\pi$, call
	the over-rotation interval $I_{\pi} = [r_{\pi}, \frac{1}{2}]$ of the
	$P$-linear map $f_P$ the \emph{over-rotation interval \textbf{of}
		$\pi$}.
	
	\begin{definition}
		A \emph{pattern} $\pi$ is called \emph{over-twist} if it does not
		\emph{force} any other \emph{pattern} of the same \emph{over-rotation
			number}.
	\end{definition}
	
	By Theorem \ref{overrotation:pairs}, an over-twist pattern has a \emph{coprime}
	\emph{over-rotation pair} (i.e., an over-rotation pair $(p,q)$ where $p$ and
	$q$ are coprime); in particular, over-twists of over-rotation
	number $\frac{1}{2}$ are of period $2$, so from now on we consider
	\emph{over-twists} of \emph{over-rotation number} distinct from
	$\frac{1}{2}$. By \cite{BM1} and by properties of forcing relation, for any
	$\frac{p}{q}\in (r_f, \frac{1}{2})$, an over-twist pattern
	of over-rotation number $\frac{p}{q}$ is exhibited by a cycle of $f$; over-twists are patterns
	that are \emph{guaranteed} to a map $f$ if the interior of $I_f$ contains the appropriate
	\emph{over-rotation number}. This can be sharpened if $f$ is
	piecewise-monotone.
	
	\begin{theorem}[\cite{B1}, \cite{blo95a}, \cite{B2}]\label{t:sharp}
		If $f:[0,1] \to [0,1]$ is a piecewise monotone continuous map with
		over-rotation interval $[r_f, \frac{1}{2}]$, then for any
		$\frac{p}{q}\in [r_f, \frac{1}{2}]$, there exists a periodic orbit $P$
		which exhibits over-twist pattern of over-rotation number
		$\frac{p}{q}$.
	\end{theorem}
	
	The over-rotation interval of an over-twist pattern $\pi$ is
	$[\rho(\pi), \frac{1}{2}]$ so that $r_{\pi}= \rho(\pi)$ holds. In
	\cite{BB1} the following version of the opposite statement is proven:
	a pattern $\pi$ with \emph{coprime
		over-rotation pair} is \emph{over-twist} if and only if $r_{\pi} =
	\rho(\pi)$. 
	
	In \cite{BS}, it was proven that for a given rational number
	$\frac{p}{q}$  there exists a unique unimodal over-twist pattern
	$\gamma_{\frac{p}{q}}$ of over-rotation number $\frac{p}{q}$; moreover, the
	dynamics of $\gamma_{\frac{p}{q}}$ was described. It was also shown that the over-rotation
	interval is a monotone function of a map considered on a wide variety of
	one-parameter families of unimodal maps. The goal of the present paper is
	to extend this result onto the family of truncations of bimodal interval maps. It is
	known (see, e.g., \cite{MT})
	that any bimodal interval map can be modeled by a truncation of a bimodal
	horseshoe map. So, we parameterize the family  of all truncations of a bimodal
	horseshoe map and then show that the set of the parameters which correspond to
	a particular to a fixed over-rotation interval is a connected subset of the parameter
	space. For this we will use the explicit description of all \emph{bimodal} over-twist
	patterns obtained in a recent paper \cite{BB2}. We discuss our plans in details below.
	
	A \emph{bimodal horseshoe
		map}  is a \emph{bimodal} map $H_2: [0,1] \to [0,1]$ defined by the following formula
	
	\begin{equation}
		\nonumber
		H_2(x)=
		\begin{cases}
			
			3x  &   \text{if $0 \leq x \leq \frac{1}{3} $}\\
			2-3x  & \text{if $\frac{1}{3} \leq x \leq \frac{2}{3}$}\\
			3x-2 & \text{if $\frac{2}{3} \leq x \leq 1 $ }\\
			
		\end{cases}
	\end{equation}
	
	Evidently, $H_2$ has cycles that exhibit all bimodal patterns.

	For $\alpha, \beta \in [0,1]$ with $ \alpha \geqslant \beta$, we call
	the interval $\xi^{max}_{\alpha} = [\frac{\alpha}{3}, \frac{2}{3} -
	\frac{\alpha}{3}]$ where $ H_2(x) \geqslant \alpha$ the  \emph{level}
	$\alpha$ \emph{max flat spot} and the interval $\xi^{min}_{\beta}
	=[\frac{2}{3} - \frac{\beta}{3}, \frac{2}{3} + \frac{\beta}{3}]$ where
	$H_2(x) \leqslant \beta$ the \emph{level} $\beta$ \emph{min flat spot}.
	The map $H_{\alpha,\beta}: [0,1] \to [0,1]$ defined by:
	
	\begin{equation}
		\nonumber
		H_{\alpha,\beta}(x)=
		\begin{cases}
			
			H(x) & \text{if $ x \notin \xi^{max}_{\alpha} \cup \xi^{min}_{\beta} $}\\
			
			\alpha  & \text{if $ x \in \xi^{max}_{\alpha} $}\\
			\beta  & \text{if $ x \in  \xi^{min}_{\beta} $ }\\
			
		\end{cases}
	\end{equation}
	
	\noindent is called the \emph{bimodal truncation} of $H$ with parameters $\alpha$ and $\beta$.
	Figure \ref{bimodal:truncation} shows the graphs of $H_2$ and $H_{\alpha,\beta}$.
	
	\begin{figure}[H]
		\caption{The graphs of the maps $H_2$ and $H_{\alpha,\beta}$ with the
			latter shown in bolder lines}
		\centering
		\includegraphics[width=1.3\textwidth]{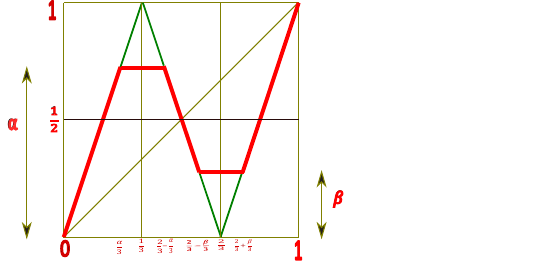}
		\label{bimodal:truncation}
	\end{figure}
	
	\noindent Recall that by \cite{MT} any bimodal map can be modeled by some
	$H_{\alpha,\beta}$.

	Let $\mathcal{H} = \{ H_{\alpha,\beta}: \alpha, \beta \in [0,1] $ $ \& $ $
	\alpha \geqslant \beta \}$
	be the family of all truncations of $H_2$. Denote the \emph{over-rotation interval}
	of $H_{\alpha,\beta}$ by $I_{\alpha,\beta}$. It is easy to see that if
	$\alpha = \beta$, $\alpha \leqslant \frac{1}{2}$, or $\beta \geqslant
	\frac{1}{2}$, then all periodic points of $H_{\alpha, \beta}$ are fixed.
	Hence, the $\omega$-limit set $\omega(x)$ is a fixed
	point for every $x$, and we are not considering such maps. Excluding these
	parameters, we define the \emph{parameter space} as $\mathcal{P}
	= \{ (\alpha, \beta) \in [0,1] \times [0,1] : \alpha \geqslant
	\frac{1}{2} \geqslant \beta \}$; $\mathcal{P}$ is a \emph{rectangle}
	with vertices $(1,0), (1, \frac{1}{2}), (\frac{1}{2}, \frac{1}{2})$ and
	$(\frac{1}{2}, 0)$. We will study $\mathcal{P}$ using the Euclidian metric.

	Let us use the following terminology to describe $\mathcal{P}$
	throughout the paper. The \emph{points} $(1,0)$ and $(\frac{1}{2},
	\frac{1}{2})$ in the \emph{parameter plane} $\mathcal{P}$ are called
	the \emph{focal point} $\mathcal{F}$ and the \emph{vertex}
	$\mathcal{V}$ respectively. Let us call the set $\mathcal{B} = \{(x, \frac12)
	|\frac{1}{2}\le x\le 1\}\cup \{(\frac12, y)|0\le y\le \frac{1}{2}\}$ the \emph{base set}.
	Similarly, call the set $\mathcal{S} = \{(x, 0)|\frac{1}{2}\le x\le 1\}\cup \{(1, y)|0
	\le y \le \frac{1}{2}\}$ the \emph{side arm}. 
	
	\begin{figure}[H]
		\caption{ The Parameter space $\mathcal{P}$ }
		\centering
		\includegraphics[width=0.7\textwidth]{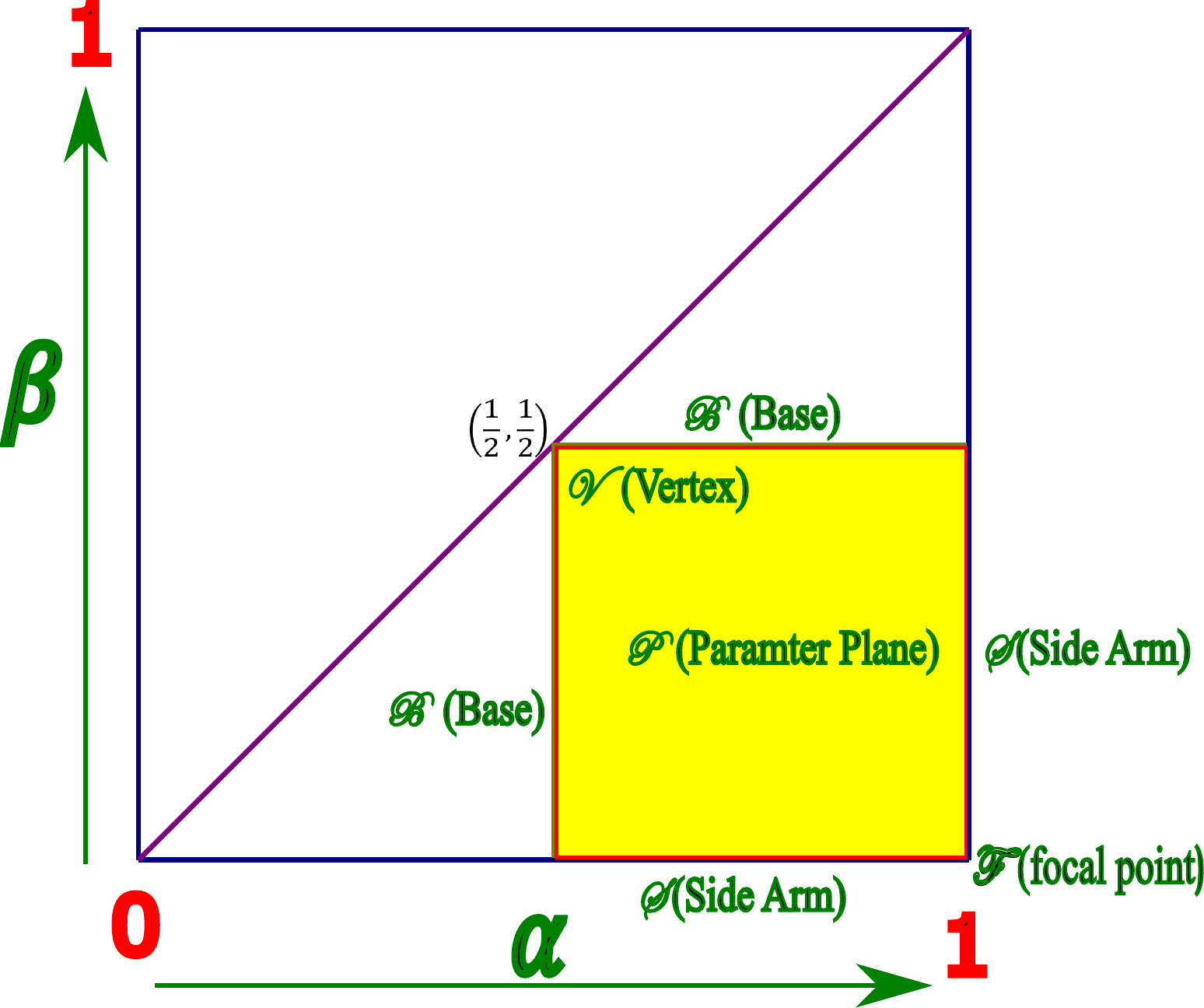}
		\label{parameter:plane:figure}
	\end{figure}
	
	Define a map $\psi : \mathcal{P} \to [0, \frac{1}{2}]$ with
	$\psi(\alpha, \beta) = \rho(I_{\alpha, \beta})$. Observe that at
	the \emph{focal point} the value of $\psi$ is the \emph{least} namely
	$0$, while at any point in the \emph{base} $\mathcal{B}$, the value of
	$\psi$ is the \emph{greatest} namely, $\frac{1}{2}$. We will call the
	set  $\psi^{-1}(\nu) = \psi_{\nu} = \{ (\alpha,\beta) \in \mathcal{P} :
	\psi( (\alpha, \beta) = \nu \}$ the \emph{Bimodal
		Iso-over-rotation-tract} corresponding to $\nu$ or the
	$\nu$-\emph{Bimodal Iso-over-rotation-tract}. The main objective of our
	paper is to show that for any $\nu \in [0, \frac{1}{2}]$ the
	$\nu$-\emph{Bimodal Iso-over-rotation-tract} is a \emph{connected set}.
	In other words the map $\psi : \mathcal{P} \to [0, \frac{1}{2}]$ is a
	\emph{monotone map}.
	
	We divide our paper into three sections. Section 1
	is our Introduction. Section 2 contains some
	preliminary ideas. Section 3 is the
	main section of the paper where we prove that \emph{Bimodal
		Iso-over-rotation-tracts} are connected.
	
	\section{PRELIMINARIES }\label{s:prelim}
	
	\subsection{Kneading Sequences}
	
	Let $f:[0,1] \to [0,1]$ be a bimodal map with
	turning points $0<c_1<c_2<1$. Call $\mathcal{S}= \{I_0, C_1, I_1, C_2, I_2\}$
	where $I_0=[0,c_1)$, $C_1=\{ c_1\}$, $I_1=(c_1, c_2)$, $C_2 =
	\{c_2\}$ and $I_2=(c_2, 1]$ the \emph{symbolic set} of $f$.
	Order its elements as $I_0<C_1<I_1<C_2<I_2$. Define a \emph{location function} $\mathcal{L} :
	[0,1] \to \mathcal{S}$ which assigns to each point $x \in [0,1]$, its
	unique \emph{location} $\mathcal{J}(x) \in \mathcal{S}$ defined by
	$x\in \mathcal{J}(x)$. Observe that if $x < y$, then
	$\mathcal{J}(x)< \mathcal{J}(y)$. Let $\mathcal{S}^{\mathbb{N}}$ be
	the set of all infinite sequences of symbols
	$(\mathcal{J}_0,\mathcal{J}_1, \mathcal{J}_2, \mathcal{J}_3, \dots )$
	where $\mathcal{J}_i \in \mathcal{S}$ for every $i$.
	
	Define a function $\mathcal{I}:[0,1]\to \mathcal{S}^{\mathbb{N}}$
	which associates with each point $x\in [0,1]$, the unique sequence
	$\mathcal{I}(x)=(\mathcal{J}(x), \mathcal{J}(f(x)),
	\mathcal{J}(f^2(x)),$ $\dots)\in \mathcal{S}^{\mathbb{N}}$. We call
	$\mathcal{J}(x)$, the \emph{itinerary} of the point $x$. Let $\sigma :
	\mathcal{S}^{\mathbb{N}}\to \mathcal{S}^{\mathbb{N}}$ be the
	\emph{shift map} defined by  $\sigma(\mathcal{J}_0,\mathcal{J}_1,
	\mathcal{J}_2, \mathcal{J}_3, \dots )= (\mathcal{J}_1, \mathcal{J}_2,
	\mathcal{J}_3, \mathcal{J}_4, \dots)$. It follows that $\mathcal{I}:
	[0,1]\to \mathcal{S}^{\mathbb{N}}$ conjugates the map $f:[0,1]\to
	[0,1]$ to the \emph{shift map} $\sigma:\mathcal{S}^{\mathbb{N}}\to
	\mathcal{S}^{\mathbb{N}}$. This allows one to model the dynamics of the
	map $f$ using the properties of the \emph{itineraries}.
	
	Define a \emph{sign} function $\Theta:\mathcal{S}\to$ $
	\{-1,0,1\}$ as follows: for any $j$, $\Theta(I_j)=+1$ if $f$ is
	\emph{increasing} on $I_j$, $\Theta( I_j)=-1 $ if $f$ is
	\emph{decreasing} on $I_j$, and $\Theta(C_j)=0$ for any $j$.
	Define a partial order $\succ$ on all
	\emph{itineraries} as follows. Let $\mathcal{J} =
	(\mathcal{J}_0,\mathcal{J}_1, \mathcal{J}_2, \mathcal{J}_3, \dots )$
	and $ \mathcal{K}= (\mathcal{K}_0,\mathcal{K}_1, \mathcal{K}_2,
	\mathcal{K}_3, \dots )$ be two \emph{itineraries}. Let $\kappa =
	\min \{i\in \mathbb{N}: \mathcal{J}_i\ne \mathcal{K}_i\}$, and set $
	\theta(\kappa-1)= \Pi_{i=1}^{\kappa-1} \Theta(\mathcal{J}_i)
	=\Pi_{i=1}^{\kappa-1}\Theta(\mathcal{K}_i)$. Then say $
	\mathcal{J}$ is \emph{stronger} than $\mathcal{K}$ and write $
	\mathcal{J}\succ \mathcal{K}$ iff $ \mathcal{J}_{\kappa} >
	\mathcal{K}_{\kappa}$ (if $\theta(\kappa-1)=+1$) \emph{or} $
	\mathcal{K}_{\kappa}>\mathcal{J}_{\kappa}$ (if
	$\theta(\kappa-1)= -1$). If $\theta(\kappa-1)=0$ then $\succ$ is \emph{not
		defined}.
	Theorem \ref{compare:itinerary} follows from \cite{MT}.
	
	\begin{theorem}\label{compare:itinerary}
		Let $x,y \in [0,1]$ with $x>y$. Then $\mathcal{I}(x) \succeq
		\mathcal{I}(y)$. Conversely, if for $x,y \in [0,1]$, we have
		$\mathcal{I}(x) \succ \mathcal{I}(y)$, then $x>y$.
	\end{theorem}
	
	The \emph{itineraries} $\mathcal{K}_j = \mathcal{I}(f(c_j)) \in
	\mathcal{S}^{\mathbb{N}}$ of the  \emph{critical points} $c_j$ for
	$j=1,2$ are called the \emph{kneading sequences} of the map $f$. The
	\emph{vector} $\oa{\mathcal{K}(f)} = (\mathcal{K}_1,
	\mathcal{K}_2)$ is called the \emph{kneading vector} of the map $f$.
	
	In this paper we will use Theorem \ref{compare:itinerary} to pinpoint
	the distribution of the periodic orbits corresponding to different
	bimodal over-twist patterns on a bimodal horseshoe map $H_2$ by
	comparing the itineraries of the points of absolute maxima and minima
	of these orbits.
	
	\subsection{Some results on Degree one Circle Maps}
	
	Recall some results on circle maps that we need (see \cite{alm00}
	for detail). Consider the unit circle $\uc$ and the natural projection $\pi:
	\mathbb{R}\to \uc$. If $f:\uc\to \uc$ is continuous, then
	there is a continuous map $F:\mathbb{R}\to \mathbb{R}$ such that $F
	\circ \pi=\pi\circ f$. Such a map $F$ is called a \emph{lifting} of $f$. It
	is unique up to translation by an integer. An integer $d$ with $F(x+1)=F(x)+d$
	for all $x\in \mathbb{R}$ is called the \emph{degree} of the map $f$ and is
	independent of the choice of $F$.
	
	Let us denoted by $\mathcal{L}_1$ the set of all liftings of continuous
	degree one maps of $\uc$ onto itself. Let $F \in \mathcal{L}_1$. We
	define \emph{upper and lower rotation numbers} of $x \in \mathbb{R}$
	for $F \in \mathcal{L}_1$ as $\overline{\rho_F}(x) = \displaystyle
	\limsup_{n \to \infty} \frac{F^{n}(x) - x}{n}$ and
	$\underline{\rho_F}(x) = \displaystyle \liminf_{n \to \infty}
	\frac{F^{n}(x) - x}{n}$ respectively.  If the numbers
	$\overline{\rho_F}(x) $ and $\underline{\rho_F}(x)$ are equal the
	common value is called the \emph{rotation number of x} for $F$ and is
	denoted by $\rho_F(x)$.
	
	\begin{theorem}[\cite{alm00}]\label{t:lift1}
		If $F \in \mathcal{L}_1'$ is a lifting of a circle map $f$ then
		$\rho_F(x)$  exists for all $x \in \mathbb{R}$ and is independent of
		$x$. Moreover, it is rational if and only if $f$ has a periodic point.
	\end{theorem}
	
	We will call this number the rotation number of $F$ and is denoted by $\rho(F)$.
	
	\begin{lemma}
		The function $\rho: \mathcal{L}_1' \to \mathbb{R}$ is continuous.
	\end{lemma}
	
	For $F \in \mathcal{L}_1$ we define maps $F_l,F_u \in \mathcal{L}_1'$ as follows:
	$F_l(x) = \inf \{ F(y) : y \geqslant x \}$ and $F_u(x) = \sup \{ F(y) : y \leqslant x \} $.
	
	\begin{theorem}[\cite{alm00}]\label{degree:one:results}
		
		Then the following statements are true:
		
		\begin{enumerate}
			\item Let $F \in \mathcal{L}_1$. Then, $F_l(x) \leqslant F(x)
			\leqslant F_u(x)$ for every $x \in \mathbb{R}$.
			\item Let $F, G \in \mathcal{L}_1$ with $F \leqslant G$, then
			$F_l \leqslant G_l$ and $F_u \leqslant G_u$.
			\item Let $F \in \mathcal{L}_1'$. Then, $F_l = F_u = F$.
			\item The maps $F \mapsto F_l$ and $F \mapsto F_u$ are
			Lipschitz continuous with constant $1$ in the sup norm.
			\item The maps $F \mapsto \rho(F_l)$ and $F \mapsto \rho(F_l)$
			are continuous.
			\item If $F_l(x) \neq F(x)$, then $x \in Const(F_l)$ where
			$Const(F_l)$ denotes the union of all open intervals on which $F_l$
			is constant.
			\item Let $F \in \mathcal{L}_1$ be a lifting of a circle map
			$f$.  Then the set of all rotation numbers of points is equal to
			the interval $[\rho(F_l), \rho(F_u)]$. Moreover, for each rational
			$a$ from this interval there is a point $x$ such that $\pi(x)$ is
			periodic for $f$ and $\rho_F(x) = a$. The interval $[\rho(F_l),
			\rho(F{_u})]$ is called the \emph{rotation interval of} $F$. We
			will denote it by $Rot(F)$.
			\item The endpoints of $Rot(F)$ depend continuously on $F \in
			\mathcal{L}_1$.
		\end{enumerate}
		
	\end{theorem}
	
	\subsection{Results on over-rotation numbers of bimodal maps \cite{BB2}}
	
	The article \cite{BB2} relates the \emph{over-rotation numbers} of bimodal interval maps and the classical rotation numbers of degree one circle maps of
	the real line. This yields a description of dynamics of
	\emph{bimodal  over-twist patterns}. Namely, given an over-rotation number $\frac{p}{q}$ there are $q-2p+1$
	\emph{bimodal over-twist patterns} $\Gamma_{r,\frac{p}{q}}$ with this over-rotation number;
	each of them can be characterized by a value of $r\in \{0, 1, 2,
	\dots, q-2p\}$. The patterns with $r=0$
	and $r=q-2p$ are \emph{unimodal} while the remaining $q-2p-1$ patterns
	are \emph{strictly bimodal}. Denote the permutation of the
	over-twist pattern $\Gamma_{r,\frac{p}{q}}$ by $\Pi_{r,
		\frac{p}{q}}$. By \cite{BB2} $\Pi_{r, \frac{p}{q}}$
	is as follows.
	
	\begin{equation}
		\Pi_{r,\frac{p}{q}}=
		\begin{cases}
			
			j+p  &   \text{if $1 \leq j \leq r $}\\
			q-j+r+1  & \text{if $r+1 \leq j \leq r+p$}\\
			2p-j+r+1 & \text{if $r+p+1 \leq j \leq r+2p$ }\\
			j-p & \text {if $r+2p+1 \leq j \leq q$}\\
		\end{cases}
	\end{equation}

	Let $P_{r,\frac{p}{q}}$ be a periodic orbit of a map $f$ which exhibits
	the pattern $\Gamma_{r, \frac{p}{q}}$. Number its $q$ points in \emph{spatial
		labelling} as $x_1<x_2<\dots<x_q$ and partition them into $4$
	disjoint parts.
	
	The first $r$ points $x_1,
	x_2, \dots, x_r$ are called \emph{red points} and are denoted by
	$\mathcal{R}_1^r, \mathcal{R}_2^r, \mathcal{R}_3^r, \dots,
	\mathcal{R}_r^r $. Under the action of $f$ the \emph{red points} are
	shifted to the right by $p$ points.
	
	The next $p$ points
	$x_{r+1}, x_{r+2},$ $ \dots, x_{r+p}$ are called \emph{green
		points} and are denoted by $\mathcal{G}_{1}^r, \mathcal{G}_{2}^r,
	\mathcal{G}_{3}^r, \dots, \mathcal{G}_{p}^r$. The \emph{green points}
	map onto the last $p$ points $x_{q-p+1}, \dots, x_{q-1}, x_{q}$ of the
	orbit with a \emph{flip}, that is orientation is 
	reversed but without any expansion.
	
	The next $p$ points
	$x_{r+p+1}, $ $ x_{r+p+2},$ $ \dots, x_{r+2p}$ are called
	\emph{pink points} and are denoted by $\mathcal{P}_1^r, \mathcal{P}_2^r,
	\mathcal{P}_3^r,$ $ \dots, \mathcal{P}_p^r $. The \emph{pink points}
	map onto the first $p$ points $x_{1}, x_{2}, $ $ \dots, x_{p}$
	of the orbit with a \emph{flip} but with no expansion.
	
	The last $q-2p-r$ points are called \emph{blue points} and
	are denoted by $\mathcal{B}_1^r,
	\mathcal{B}_2^r, \mathcal{B}_3^r, \dots, \mathcal{B}_{q-2p-r}^r $. They
	are shifted to the \emph{left} by $p$ points.
	
	If $r=0$, we have no \emph{red points}; if $r=
	q-2p$, we have no \emph{blue points}. For all other values of $r$, we
	have points of all colors.
	
	The dynamics of $P_{3, \frac{3}{11} }$ is depicted in the Figure \ref{f:bimodal-2}.
	
	\begin{figure}[H]
		\caption{ The Bimodal over-twist pattern $\Gamma_{3,\frac{3}{11}}$ }
		\centering
		\includegraphics[width=0.7\textwidth]{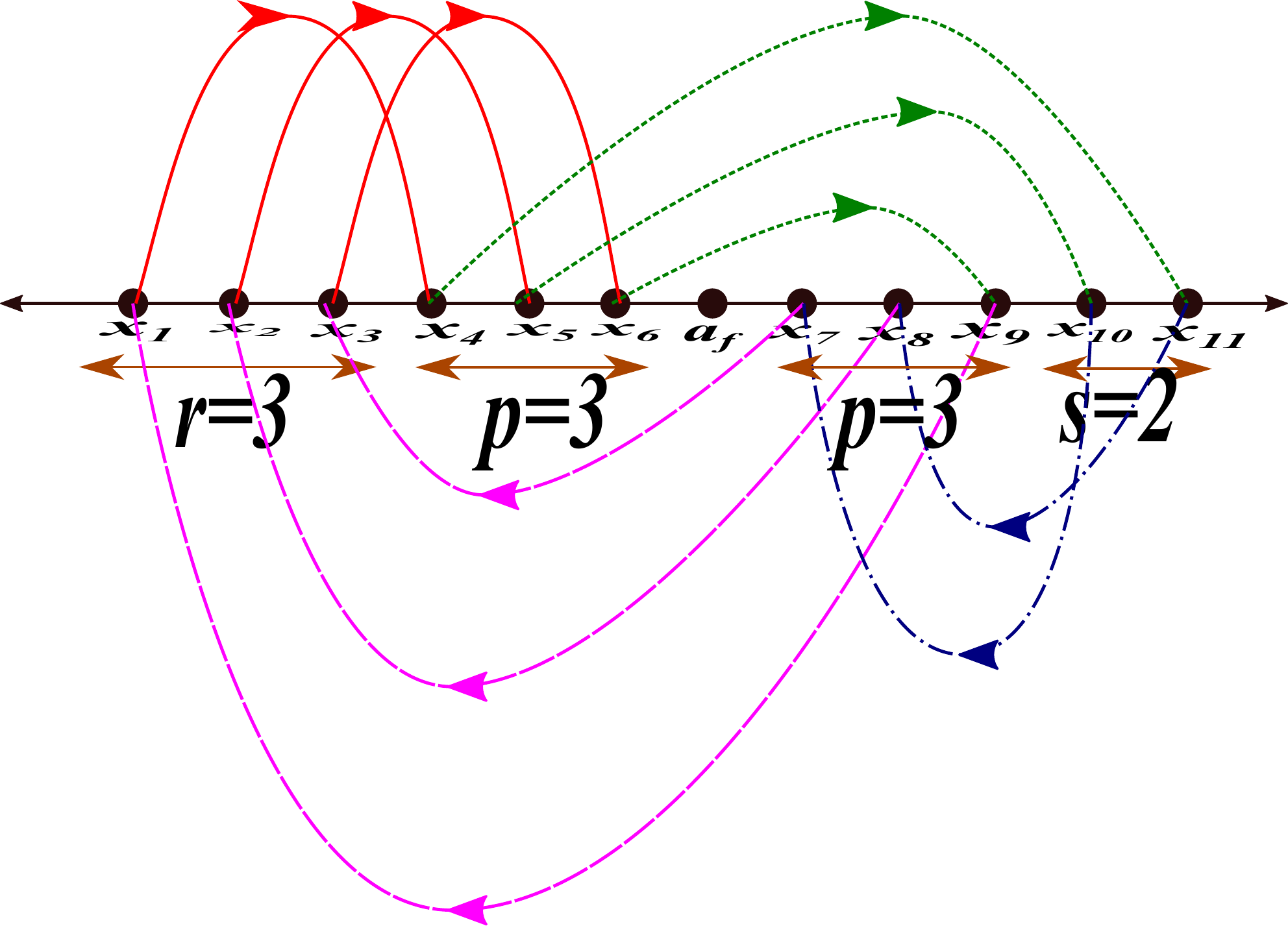}\label{f:bimodal-2}
	\end{figure}
	
	\subsection{Some results from Algebraic Topology}

	We assume knowledge of basic concepts such as connectedness and path-connectedness of topological spaces.
	
	If $f: X\to Y$ and $g: X\to Y$ are two continuous functions
	then $f$ and $g$  are said to be \emph{homotopic} if there exists a continuous
	function $H: (X, [0, 1])\to X$ such that $H(x, 0)=f(x)$, $H(x, 1)=g(x)$.
	Then we write $f\cong f'$ ; the map $H$
	is called a \emph{homotopy} between $f$ and $g$. Let $(X, \tau)$ be a
	topological space and $A\subset X$. A \emph{retraction} of $X$ onto $A$ is a
	continuous map $r: X\to A$ such that $r|_A = id_A$. If such a map
	$r$ exists then $A$ is called a \emph{retract} of $X$.
	
	\begin{definition}
		Let $(X, \tau)$ be a topological space and $A \subset X$. We say that$A$
		is a \emph{deformation retract} of $X$ if there exists a homotopy between the
		identity map on $X$ and the map that carries all of $X$ onto
		$A$ such that each point of $A$ remains fixed during the homotopy.
		
		In other words, there is a homotopy $H: X\times I \to X$ such
		that:
		\begin{enumerate}
			\item $H(x,0) = id_X$ $\forall x \in X$.
			\item $H(x,1) \in A$
			\item $H(a,t) =a $ $\forall t \in I$ and $\forall a \in A$.
		\end{enumerate}
		
	\end{definition}
	
	The map $r:X \to A$ defined by $r(x)=H(x,1)$ is a retraction of $X$
	onto $A$ and  $H$ is a homotopy between $id_X$ and the map $ j \circ r$
	where $ j: A \to X$ is the inclusion map defined by $j(a)=a$ for all $a
	\in A$.
	
	\begin{theorem}\label{deformation:connected}
		Let $(X,\tau)$ be a path-connected space and let $A \subset X$ be a
		deformation  retract of $X$. Then, $A$ is path-connected.
	\end{theorem}
	
	\begin{proof}
		Take $a, b \in A$. Since $(X, \tau)$ is path-connected, then there exists
		a continuous  map $\xi:[0,1]\to X$ with $\xi(0)=a$ and
		$\xi(1)=b$. Since $A$ is a deformation retract of $X$, there exists
		a continuous map $H:(X, [0,1])\to X$ such that $H(x,0)=id_X$ and
		$H(x,1)|_A =id_A$. Define $r:X\to A$ such that $r(x)=H(x, 1)$.
		Then $r|_A = id_A$.
		
		Define $\chi:[0,1]\to A$ by $\chi=r\circ \xi$. Then $\chi(0)=a$ and
		$\chi(1)=b$ and $\chi$ is continuous. Thus, $\chi$
		is a path in $A$ connecting $a$ and $b$. So, $A$ is a path-connected.
	\end{proof}
	
	\section{CONNECTEDNESS OF BIMODAL ISO-OVER-ROTATION-TRACTS}\label{bimodal:trunc}
	
	Let us outline our approach. Given a rational number $\frac{p}{q}$, we first
	describe the set $Z_{\frac{p}{q}}$ of all points $(\alpha,\beta)$
	in the parameter space $\mathcal{P}$ such that
	$\psi((\alpha,\beta)) = \frac{p}{q}$ and one of the flat spots $\xi^{max}_{\alpha}$
	and  $\xi^{min}_{\beta}$ exhibits an over-twist pattern of over-rotation
	number $\frac{p}{q}$; this set
	is a continuous broken line (a \emph{staircase}) in $\mathcal{P}$.
	Then we show that $Z_{\frac{p}{q}}$
	\emph{partitions} $\mathcal{P}$ into two
	\emph{path-connected components}: a \emph{component} $\mathcal{G}_{\frac{p}{q}}$
	where $\psi((\alpha,\beta))\geqslant \frac{p}{q}$ and a \emph{component}
	$\mathcal{U}_{\frac{p}{q}}$ where $\psi((\alpha,\beta))< \frac{p}{q}$.
	Then we show that the $\frac{p}{q}$-\emph{Bimodal Iso-over-rotation-tract}
	$\mathcal{I}_{\frac{p}{q}}$ is a \emph{deformation retract} of the component
	$\mathcal{G}_{\frac{p}{q}}$. Since $\mathcal{G}_{\frac{p}{q}}$ is path-connected,
	then, by Theorem \ref{deformation:connected}, $\mathcal{I}_{\frac{p}{q}}$ is path-connected.
	Finally we prove that $\psi : \mathcal{P}  \to [0, \frac{1}{2}]$ is continuous and
	then use it to prove connectedness of $\mu-\emph{Bimodal Iso-over-rotation-tract}$
	for an irrational number $\mu\in [0, \frac{1}{2}]$.
	

For two sets $A,B \subset \mathbb{R}^2$, $d(A,B)=\inf
\{d(x,y):x\in A $ $\& $ $y\in B\}$.
	
\begin{lemma}\label{distance:pairs}
If $(\alpha,\beta)$ and  $(\gamma,\delta)$ are points in  $\mathcal{P} $ co-linear with the focal point $\mathcal{F}$ such that
$d(\mathcal{F},(\alpha,\beta))$ $ \geqslant d(\mathcal{F},
(\gamma,\delta))$ then $\psi((\alpha,\beta))  \geqslant \psi((\gamma,\delta)) $.
\end{lemma}
	
\begin{proof}
Since, $(\alpha,\beta)$ and  $(\gamma,\delta)$ are points in  $\mathcal{P}$ co-linear with the focal point $\mathcal{F}$ with
$d(\mathcal{F}, (\alpha,\beta))$ $ \geqslant d(\mathcal{F},
(\gamma,\delta))$, it follows that $\gamma \geqslant \alpha$
and $\delta \leqslant \beta$. Thus $\xi_{\gamma}^{max} \subseteq
\xi_{\alpha}^{max}$, $\xi_{\delta}^{min} \subseteq
\xi_{\beta}^{min}$, and hence $I_{(\alpha,\beta)} \subseteq I_{(\gamma,\delta)}$
and $\psi((\alpha,\beta))\geqslant \psi((\gamma,\delta))$ as desired.
\end{proof}
	
\begin{definition}\label{leading:set:rational}
Let $\frac{p}{q}$ be a given rational number where $p$ and $q$ are
\emph{coprime}. The \emph{leading set} $Z_{\frac{p}{q}}$ is
the collection of points
$(\alpha,\beta) \in \mathcal{P}$ which satisfies the
following:
		
\begin{enumerate}
\item at least one of \emph{flatspots} $\xi_{\alpha}^{max}$ and
$\xi_{\beta}^{min}$ exhibits an \emph{over-twist pattern} with
\emph{over-rotation number} $\frac{p}{q}$.
\item $\psi((\alpha,\beta)) = \frac{p}{q}$.
\end{enumerate}	
\end{definition}
	
	We now describe the structure of the leading set $Z_{\frac{p}{q}}$  for
	a given rational number $\frac{p}{q}$ where $p$ and $q$ are coprime.
	For this we first study the distribution of the  bimodal over-twist
	periodic orbits in a bimodal horseshoe map $H_2$.
	
	For any rational $\frac{p}{q}$ with \emph{coprime} $p$ and $q$, there are $q-2p$ bimodal over-twist patterns $\Gamma_{r,\frac{p}{q}}$ for $r\in \{0, 1, 2, \dots, q-2p\}$. Out of these the patterns $\Gamma_{0,\frac{p}{q}}$ and $\Gamma_{q-2p,\frac{p}{q}}$ are unimodal and the remaining $q-2p-2$ patterns are strictly bimodal. Recall that in the context of the map $H_2$ we have
$I_0=[0, \frac13],$ $I_1=[\frac13, \frac23],$ $I_2=[\frac23, 1]$.
	
The map $H_2$ has a unique cycle $P_{0,\frac{p}{q}}$
such its point of \emph{absolute maximum} $M_0$ and its point of \emph{absolute minimum} $m_0$ lies in $I_1$.
Similarly, $H_2$ has a unique periodic orbit $P_{q-2p,\frac{p}{q}}$ such its point of \emph{absolute maximum} $M_{q-2p}$ and its point of \emph{absolute minimum} $m_{q-2p}$ lies in $I_1$.

Finally, for each $r\in \{1, 2, \dots, q-2p-1\}$, $H_2$ has a unique periodic orbit $P_{r,\frac{p}{q}}$
	such its point of \emph{absolute maximum} $M_r$ lies in $I_1$ and its point of \emph{absolute minimum} $m_r$ lies in $I_2$.

 Let $\mathcal{P}_{\frac{p}{q}}$ be the collection of these $q-2p+1$ cycles. Let us study the location of these cycles on $H_2$. For this we introduce a \emph{partial ordering}
	$\triangleleft$ in $\mathcal{P}_{\frac{p}{q}}$ as follows.  For $r_1,
	r_2\in \{0, 1, 2, \dots, q-2p\}$,  write $P_{r_2,\frac{p}{q}}$ $
	\triangleleft$ $P_{r_1,\frac{p}{q}}$ and say $P_{r_2, \frac{p}{q}}$
	\emph{floats above}  $P_{r_2,\frac{p}{q}}$ in $H_2$ iff $H_2(M_{r_2}) >
	H_2(M_{r_1}) $ and $H_2(m_{r_2}) > H_2(m_{r_1}) $ where $M_{r_1}$ and
	$M_{r_2}$ are the points of \emph{absolute maximum} and  $m_{r_1}$ and
	$m_{r_2}$ be the points of \emph{absolute minimum} of these orbits
	respectively.

 Given  a periodic orbit $P_{r, \frac{p}{q}}$  it is easy to see that 
	that all its \emph{red} points (if any) lies in the lapse $I_0$, all its \emph{green} and \emph{pink} points lies in the lapse $I_1$ and all its \emph{blue} points (if any ) lies in the lapse $I_2$.
	
	\begin{lemma}\label{floating:above:1}
		For every $ 1 \leqslant k \leqslant q-2p-2$, $P_{k,\frac{p}{q}} $ $
		\triangleleft $ $ P_{k+1, \frac{p}{q}}$. 	
	\end{lemma}
	
	\begin{proof}
		Since, $1 \leqslant k \leqslant q-2p-2$, both $P_{k,\frac{p}{q}} $ and
		$P_{k+1,\frac{p}{q}}$ contains points of all colors. Let us compare the
		itineraries of $m_k$ and $m_{k+1}$. They are both \emph{pink} points in
		their respective orbits. On application of $H_2$ both of these points will map to the
		first red point $\mathcal{R}_1^{k}$ and $\mathcal{R}_1^{k+1}$ of their
		respective orbits, then to a \emph{green point}, then to a \emph{blue
			point} and then finally to another \emph{pink point} of their
		respective orbits. Let us call this one \emph{round} of its orbit.

Both $m_k$ and $m_{k+1}$ will keep on performing such \emph{rounds} and their itineraries will remain the same until the auspicious moment when $m_{k+1}$ hits the
		last \emph{red point} $\mathcal{R}_k^{k+1}$ of its orbit. Since the number of \emph{red points} in $P_{k,\frac{p}{q}}$ is one less than
		that of $P_{k+1, \frac{p}{q}}$, at that moment $m_k$ will hit the first
		green point $\mathcal{G}_1^{k}$ of its orbit. Observe that
		$P_{k,\frac{p}{q}}$ attains its absolute maximum at
		$\mathcal{G}_1^{k}$, that is, $M_k = \mathcal{G}_1^{k}$. So, on the
		next step while $m_{k+1}$ will hit a green point with corresponding
		symbol $I_1$, $m_k$ will hit the last \emph{blue} point
		$\mathcal{B}^k_{q-2p-k}$ with symbol $I_2$.

		Finally, observe that in each
		\emph{round} both $m_k$ and $m_{k+1}$ reach the decreasing lapse $I_2$ of
		$H_2$ an even number of times. So, it follows that the itinerary $H(m_k)$
		is \emph{stronger} than that of  $H(m_{k+1})$. So, by Theorem \ref{compare:itinerary} it follows that $H(m_{k+1}) > H(m_k)$.
		
		Similarly, it can be seen that the itinerary of   $H_2(M_{k+1})$ is
		\emph{stronger} that for  $H_2(M_{k})$ and hence by
		Theorem \ref{compare:itinerary} the result follows. 	
	\end{proof}

	\begin{lemma} \label{floating:above:2}
		$\Gamma_{0,\frac{p}{q}} $  $ \triangleleft $ $\Gamma_{1,\frac{p}{q}}$.
	\end{lemma}
	
	\begin{proof}
		$P_{0,\frac{p}{q}}$ has no red points. Comparing the itineraries of
		$m_0$ and $m_1$, it is easy to see that
		$\mathcal{I}(H_2(m_0))=(I_1,I_2,\dots)$ while
		$\mathcal{I}(H_2(m_1))=(I_0,I_1,\dots)$. Thus, clearly $\mathcal{I}(H_2(m_0))$ is \emph{stronger} (in the sense of $\succ$ )  than $\mathcal{I}(H_2(m_1))$. Thus, by Theorem \ref{compare:itinerary}, $H_2(m_0) > H_2(m_1)$.
		
		The \emph{itineraries} of  $M_0$ and $M_1$ will remain the same till the moment $M_0$ hits the first \emph{blue point} from the left in the spatial labelling :$\mathcal{B}_1^0$ of its  orbit $P_{0,\frac{p}{q}}$. At that point the point $M_1$ maps to the last pink point $\mathcal{P}_q^1$ of its orbit $P_{1,\frac{p}{q}}$ in the spatial labelling. Next application of $H_2$ maps $M_0$ to some \emph{pink} point of its orbit while $M_1$ maps to the first \emph{red point} $\mathcal{R}_1^1$ of
		its orbit. Arguing as in Lemma \ref{floating:above:1} it follows that
		 $H_2(M_0)> H_2(M_1)$
	\end{proof}
	
	Similarly we can prove that:
	
	\begin{lemma} \label{floating:above:3}
		$\Gamma_{q-2p-1,\frac{p}{q}} $  $ \triangleleft $
		$\Gamma_{q-2p,\frac{p}{q}}  $ 	
	\end{lemma}

	Using Lemma \ref{floating:above:1}, \ref{floating:above:2},
	\ref{floating:above:3}  we
	can describe the configuration of the \emph{leading set}
	$L_{\frac{p}{q}}$ in $\mathcal{P}$. For this we need the following
	definition:
	
	\begin{definition}\label{staircase}
		Let $A=[a,b] \times [a,b]$ where $a,b \in \R$. Let $P_1 = \{ a=x_0,
		x_1, \dots, x_{n+1} = b\}$ and $P_2 = \{ a=y_0, y_1, \dots, y_{n+1} =
		b\}$ be two \emph{partitions} of $[a,b]$.
		
		By a $n$-\emph{staircase} $\mathcal{S}_n$ in $A$ with $2n$ steps denoted by $S_n$  we mean:
		
		\begin{enumerate}
			\item A set of $n$ \emph{vertical line segments} : $\{x_1 \} \times
			[y_1, y_2], $ $\{x_2 \} \times [y_2, y_3], $ $ \dots, \{x_n \}
			\times [y_n, y_{n+1} = b]$ (\emph{called rises}).
			\item A set of $n$ \emph{horizontal line segments} :$[a=x_0,x_1]
			\times \{y_1 \}, $ $[x_1, x_2] \times \{y_2\}, \dots, [x_{n-1}
			, x_{n}] $ $\times \{ y_n \}$ (\emph{called treads}). 	
		\end{enumerate}
		
		$x_1 - x_0$ is called the length of the \emph{initial tread} while
		$y_{n+1} - y_n $ is called the length of the \emph{final rise}. The
		lengths $ \Delta x_i = x_{i+1}-x_i$ and $\Delta y_i = y_{i+1}-y_i$ are
		called the lengths of the \emph{ith tread} and \emph{ith rise}
		respectively for $1 \leqslant i \leqslant n$.
		
		A staircase $S_n$ is said to be \emph{symmetrical} if $\Delta x_{i-1} =
		\Delta y_{n+1-i}$ for every $ i \in \{1,2 \dots n+1\}$.
	\end{definition}
	\begin{theorem}\label{L:structure:rational}
		$Z_{\frac{p}{q}}$ is a symmetrical staircase with $2(q-2p+1)$ steps in
		$\mathcal{P}$ with initial tread and final rise lying in the \emph{side
			arm} $\mathcal{S}$.
	\end{theorem}
	
	\begin{proof}
		The location of the points $m_r$ and $M_r$ of the over-twist periodic
		orbits $P_{r,\frac{p}{q}}$ in $H_2$ is elucidated by  Lemma
		\ref{criticalpoints:overtwist:location}. To describe $Z_{\frac{p}{q}}$
		we take an arbitrary point $(\alpha, \beta)$ in $\mathcal{P}$ and
		waggle it in $\mathcal{P}$ and look for moments when at least one of
		$\xi_{\alpha}^{max}$ and $\xi_{\beta}^{min}$ is part of
		$P_{r,\frac{p}{q}}$ for some $r$. First we take $\alpha =1 $ and $\beta
		= \frac{1}{2}$. At this point the map $H_{\alpha, \beta}$ has no
		\emph{non-fixed periodic} point and hence its \emph{dynamics} is not
		interesting. No,we slowly \emph{decrease} the value of $\beta $ keeping
		$\alpha$ fixed at $1$. There will be some unique value $\beta_0$ of
		$\beta$ when $\xi_{\beta}^{min}$ will hit $H_2(m_0)$. At this moment
		$\xi_{\beta}^{min}$ is part of the orbit $P_{0,\frac{p}{q}}$ while
		$\xi_{\alpha}^{max}$ is not. Now, we fix $\beta = \beta_0$ and slowly
		decrease $\alpha$. For some time only $\xi_{\beta}^{min}$ will be part
		of $P_{0,\frac{p}{q}}$, but then we will encounter some unique value
		$\alpha = \alpha_0$ of $\alpha$ when $\xi_{\alpha}^{max}$ hits
		$H_2(M_0)$. At the point $(\alpha_0,\beta_0)$ both $\xi_{\alpha}^{max}$
		and $\xi_{\beta}^{min}$ will be part of $P_{0,\frac{p}{q}}$. Now, we
		keep $\alpha = \alpha_0$ and slowly decrease $\beta$. As $\beta$ moves
		away from $\beta_0$, $\xi_{\beta}^{min}$ will no longer be part of
		$P_{0,\frac{p}{q}}$ and only $\xi_{\alpha}^{max}$ will remain part of
		$P_{0,\frac{p}{q}}$. This will continue for some time. Then for some
		$\beta = \beta_1$, $\xi_{\beta}^{min}$ will hit $H_2(m_0)$. At that
		moment $\xi_{\beta}^{min}$ and $\xi_{\alpha}^{max}$ will both be
		periodic but parts of different orbits. While $\xi_{\beta}^{min}$ is
		now part of $P_{1,\frac{p}{q}}$, $\xi_{\alpha}^{max}$ will remain part
		of $P_{0,\frac{p}{q}}$. Now, keeping keeping $\beta$ fixed at $\beta_1$
		we decrease $\alpha$ from $\alpha_0$. Now, the orbit
		$P_{0,\frac{p}{q}}$ is completely cut off from our map
		$H_{\alpha,\beta}$ and for some time only $\xi_{\beta}^{min}$ will
		remain part of $P_{1,\frac{p}{q}}$ until for some value of $\alpha_1$
		of $\alpha$, $\xi_{\alpha}^{max}$ hits $H_2(M_1)$ and both \emph{flat
			spots} become part of $P_{1,\frac{p}{q}}$. Repeating the process
		$q-2p+1$ times( for each of the $q-2p+1$ \emph{over-twist periodic
			orbits} $P_{r,\frac{p}{q}}$ ), it is easy to see that $Z_{\frac{p}{q}}$
		is a \emph{staircase} with $2(q-2p+1)$ steps. The \emph{symmetry} of
		the \emph{staircase} follows from the \emph{symmetry} of $H_2$.	
	\end{proof}
	
	If $ (\alpha,\beta) \in Z_{\frac{p}{q}}$, then by definition $
	\psi((\alpha,\beta))= \frac{p}{q}$.  The next illustrates what happens
	if $ \psi((\alpha,\beta))= \frac{p}{q}$ but $ (\alpha,\beta) \notin
	Z_{\frac{p}{q}}$:
	
	\begin{lemma}\label{kappa}
		If $ \psi((\alpha,\beta))= \frac{p}{q}$ but $ (\alpha,\beta) \notin
		Z_{\frac{p}{q}}$,  then there exists $\kappa_1, \kappa_2 \geq 0$ such
		that $ (\alpha - \kappa_1, \beta + \kappa_2) \in Z_{\frac{p}{q}}$.
	\end{lemma}
	
	\begin{proof}
		Since, $\psi(\alpha,\beta)=\frac{p}{q}$, so $H_{\alpha,\beta}$ has  a
		periodic orbit $P$ which exhibits an over-twist pattern corresponding
		to over-rotation number $\frac{p}{q}$. Also, since, $ (\alpha, \beta)
		\notin Z_{\frac{p}{q}}$, so $P$ lies in the interior of
		$[0,1]-(\xi_{max}^{\alpha} \cup \xi_{min}^{\beta} )$. Thus, there exits
		$\kappa_1, \kappa_2 \geq 0$ such that one of the flat spots
		$\xi_{max}^{\alpha - \kappa_1}$ and $\xi_{min}^{\beta - \kappa_2}$ hits
		a point $P$ while the other flat spot stays outside all points of $P$.
		Thus, $ (\alpha - \kappa_1, \beta + \kappa_2) \in Z_{\frac{p}{q}}$. 	
	\end{proof}
	
	\begin{lemma}\label{leading:order:rational}
		If $\frac{p}{q} , \frac{r}{s} \in \mathbb{Q} $, $
		g.c.d(p,q)=g.c.d(r,s)=1$  and $\frac{p}{q} < \frac{r}{s}$, then
		$d(\mathcal{F},Z_{\frac{p}{q}}) < d(\mathcal{F},Z_{\frac{r}{s}}) $. 	
	\end{lemma}
	
	\begin{proof}
		We first note that $\frac{p}{q} \neq \frac{r}{s}$, so  $
		Z_{\frac{p}{q}} \cap Z_{\frac{r}{s}} \neq 0$, since
		otherwise $\psi$ would become multi-valued at the point of
		intersection. So, either $d(\mathcal{F},Z_{\frac{p}{q}}) <
		d(\mathcal{F},Z_{\frac{r}{s}}) $  or  $d(\mathcal{F},Z_{\frac{p}{q}}) >
		d(\mathcal{F},Z_{\frac{r}{s}}) $. Suppose if possible
		$d(\mathcal{F},Z_{\frac{p}{q}}) > d(\mathcal{F},Z_{\frac{r}{s}}) $.
		Choose a point $(\alpha, \beta) \in Z_{\frac{p}{q}}$ and $(\gamma,
		\delta) \in Z_{\frac{r}{s}}$ such that the points $(\alpha, \beta)$, $(\gamma,
		\delta)$ and $\mathcal{F}$ are \emph{co-linear}. Then , $d(\mathcal{F},(\alpha,\beta) >
		d(\mathcal{F},(\gamma,\delta)) $. By Lemma \ref{distance:pairs}
		$H_{\alpha,\beta}$ \emph{forces} $H_{\gamma,\delta}$. Hence,
		$\psi((\alpha,\beta)) = \frac{p}{q} \geqslant \psi((\gamma, \delta))=
		\frac{r}{s}$ which is a contradiction.
	\end{proof}
	
	To show that the $\frac{p}{q}$-\emph{Iso-over-rotation-tract} $\mathcal{I}_{\frac{p}{q}}$ is connected, we define two more sets:
	
	\begin{definition}
		Let $\mathcal{G}_{\frac{p}{q}}$ be the set of all points $(\alpha,\beta)$ in the parameter plane $\mathcal{P}$ such that
		$ \psi((\alpha,\beta))  \geqslant \frac{p}{q}$.
	\end{definition}

	\begin{definition}
		Let $\mathcal{U}_{\frac{p}{q}}$ be the set of all points $(\alpha,\beta)$ in the parameter plane $\mathcal{P}$ such that
		$
		\psi((\alpha,\beta))  < \frac{p}{q}$.
	\end{definition}

	\begin{theorem}\label{loc}
		$\mathcal{U}_{\frac{p}{q}}$ is precisely the component of
		$\mathcal{P} - Z_{\frac{p}{q}}$ containing the focal point
		$\mathcal{F}$.
	\end{theorem}
	
	\begin{proof}
		Let $(x,y)$ be any point in the component of
		$\mathcal{P} - Z_{\frac{p}{q}}$  containing the focal point
		$\mathcal{F}$. Then clearly, $d((x,y), \mathcal{F}) <
		d(Z_{\frac{p}{q}}, \mathcal{F})$. So, for every $(s,t) \in
		Z_{\frac{p}{q}}$, we have $d((s,t), \mathcal{F}) > d((x,y),
		\mathcal{F})$. Hence, $\psi((x,y)) \leqslant \frac{p}{q}$. If
		$\psi((x,y)) = \frac{p}{q}$ and $(x,y) \notin Z_{\frac{p}{q}}$, then by
		Lemma \ref{kappa}, there exists $\kappa_1, \kappa_2 \geq 0$ such that $
		(x - \kappa_1, y + \kappa_2) \in Z_{\frac{p}{q}}$. Now, clearly, $ d((x
		- \kappa_1, y + \kappa_2), \mathcal{F}) < d((x,y), \mathcal{F}) $ which
		is a contradiction. Thus,  $\psi((x,y)) < \frac{p}{q}$.
		
		Now, suppose that for some $(x,y) \in \mathcal{P}$ we have $\psi((x,y))
		< \frac{p}{q}$. We claim that $d((x,y), \mathcal{F}) <
		d(Z_{\frac{p}{q}}$. By way of contradiction, suppose $d((x,y),
		\mathcal{F}) < d(Z_{\frac{p}{q}}, \mathcal{F})$. Then, there exists
		$(w,z) \in Z_{\frac{p}{q}}$ such that $d((x,y), \mathcal{F}) > d((w,z),
		\mathcal{F})$. But, $\psi((x,y)) < \frac{p}{q} = \psi((w,z))$ which
		contradicts Lemma \ref{distance:pairs}. This proves the theorem.
	\end{proof}
	
	\begin{cor}\label{lower:limit}
		$\mathcal{U}_{\frac{p}{q}}$ is path-connected.
	\end{cor}
	
	\begin{proof}
		From Theorem \ref{loc} it follows that $\mathcal{U}_{\frac{p}{q}}$ is
		precisely the component of $\mathcal{P} - Z_{\frac{p}{q}}$ containing
		the \emph{focal point} $\mathcal{F}$. Since, $\mathcal{P}$ is
		path-connected and from Theorem \ref{L:structure:rational} it follows
		that $Z_{\frac{p}{q}}$ is path-connected, so
		$\mathcal{U}_{\frac{p}{q}}$ is path-connected. 	
	\end{proof}	
	
	\begin{lemma}\label{upper:limit}
		$\mathcal{G}_{\frac{p}{q}}$ is path-connected. 	
	\end{lemma}
	
	\begin{proof}
		Clearly, $\mathcal{G}_{\frac{p}{q}}$ = $\mathcal{P} -
		\mathcal{U}_{\frac{p}{q}}$ and from Corollary \ref{lower:limit}
		$\mathcal{G}_{\frac{p}{q}}$ is path-connected. Hence, the result
		follows.
	\end{proof}	
	
	From Theorem \ref{deformation:connected}, it follows that to show that
	the \emph{Iso-over-rotation-tract} $\mathcal{I}_{\frac{p}{q}}$ is
	connected, it is sufficient to show that $\mathcal{I}_{\frac{p}{q}}$ is
	a deformation retract of $\mathcal{G}_{\frac{p}{q}}$. This is exactly
	what we will show next.
	
	\begin{theorem}\label{greater:deformation}
		For a given rational number $\frac{p}{q}$  where $p$ and $q$ are
		coprime, $\mathcal{I}_{\frac{p}{q}}$ is a deformation retract of
		$\mathcal{G}_{\frac{p}{q}}$.	
	\end{theorem}	
	
	\begin{proof}
		Take any point $(x,y) \in \mathcal{G}_{\frac{p}{q}}$. Then  $\psi((x,y)
		\geqslant \frac{p}{q}$. We consider the \emph{path} $\Phi_{(x,y)} :
		[0,1] \to \mathcal{P}$ defined by $\Phi_{(x,y)}(t) =  (x(1-t)+t,
		y(1-t))$ (which is the \emph{directed line segment} in $\mathcal{P}$
		joining $(x,y)$ to the \emph{focal point} $\mathcal{F}$). We now start
		to move the point $(x,y)$ along the path  $\Phi_{(x,y)}(t)$ towards the
		focal point $\mathcal{F}$ a study happens with the  value of $\psi$ as
		we move along. By Lemma \ref{distance:pairs} it follows that the value
		of $\psi$ decreases monotonically as we move along the path
		$\Phi_{(x,y)}$ towards the focal point $\mathcal{F}$. So, the set $
		\mathcal{T}_{\frac{p}{q}} = \{ t \in [0,1] | \psi(\Phi_{(x,y)}(t)) =
		\frac{p}{q} \}$ is non-empty and bounded below. Let $
		\tau_{\frac{p}{q}} = inf \mathcal{T}_{\frac{p}{q}}  $.
		
		Consider the function $\mathcal{D}_{\frac{p}{q}} :
		\mathcal{G}_{\frac{p}{q}} \times [0,1] \to \mathcal{G}_{\frac{p}{q}}$
		defined by:
		
		\begin{equation}
			\nonumber
			\mathcal{D}_{\frac{p}{q}}((x,y),t)=
			\begin{cases}
				
				\Phi_{(x,y)}(t)  &   \text{if $ t < \tau_{\frac{p}{q}} $}
				
				\\
				
				\Phi_{(x,y)}(\tau_{\frac{p}{q}}) & \text{if $ t \geqslant \tau_{\frac{p}{q}} $}\\	
			\end{cases}
		\end{equation}
		
		$\mathcal{D}_{\frac{p}{q}}$ is clearly continuous and is thus a
		deformation retract of $\mathcal{G}_{\frac{p}{q}}$ onto
		$\mathcal{I}_{\frac{p}{q}}$.
	\end{proof}	
	
	\begin{cor}\label{rational:connected}
		For a given rational number $\frac{p}{q}$  where $p$ and $q$ are
		coprime,  the Iso-over-rotation-tract $\mathcal{I}_{\frac{p}{q}}$ is
		connected.	
	\end{cor}
	
	\begin{proof}
		$\mathcal{I}_{\frac{p}{q}}$ is a deformation retract of
		$\mathcal{G}_{\frac{p}{q}}$	 From Theorem \ref{greater:deformation},
		for a given rational number $\frac{p}{q}$  where $p$ and $q$ are
		coprime, $\mathcal{I}_{\frac{p}{q}}$ is a deformation retract of
		$\mathcal{G}_{\frac{p}{q}}$. Also, from Lemma \ref{upper:limit}
		$\mathcal{G}_{\frac{p}{q}}$ is path-connected. Thus, from Theorem
		\ref{deformation:connected} the result follows.
	\end{proof}	
	
	Consider the map $\psi : \mathcal{P} \to [0, \frac{1}{2}]$. Combining
	results  Theorem \ref{degree:one:results} and \cite{BB2}
	we have:

	\begin{theorem}\label{psi:continuous}
		The map $\psi : \mathcal{P} \to [0, \frac{1}{2}]$ is continuous.
	\end{theorem}

	From Corollary \ref{rational:connected} for each rational number
	$\frac{p}{q} \in [0, \frac{1}{2}]$, $\psi^{-1}(\frac{p}{q})$ is
	connected.

	Now, take an irrational number $\mu \in [0, \frac{1}{2}]$.
	We now show that $\psi^{-1}(\mu)$ is also connected.
	
	\begin{theorem}\label{irrational:connected}
		For each irrational number $\mu \in [0, \frac{1}{2}]$ the
		Iso-over-rotation-tract $\mathcal{I}_{\mu} =  \psi^{-1}(\mu)$ is
		connected.
	\end{theorem}
	
	\begin{proof}
		Consider the sets $\mathcal{G}_{\mu} = \displaystyle
		\bigcup_{\frac{p}{q} > \mu} \mathcal{G}_{\frac{p}{q}}$ and
		$\mathcal{U}_{\mu} = \displaystyle \bigcup_{\frac{p}{q} < \mu    }
		\mathcal{U}_{\frac{p}{q}}$. It follows from Theorem \ref{lower:limit}
		and Corollary \ref{upper:limit} that the sets $\mathcal{G}_{\mu}$ and
		$\mathcal{U}_{\mu}$ are connected. Hence, $\mathcal{A} = \mathcal{P} -
		(\mathcal{G}_{\mu} \cup  \mathcal{U}_{\mu})$ is connected. From Theorem
		\ref{psi:continuous} it readily follows that $\mathcal{A}$ is precisely
		the set of points $(x,y)$ of $\mathcal{P}$ such that $\psi((x,y))= \mu$
		and hence the result follows.
	\end{proof}
	
	From Corollary \ref{rational:connected} and Theorem
	\ref{irrational:connected} we finally have the main result of this
	paper:
	
	\begin{cor}\label{main:result}
		The map $\psi : \mathcal{P} \to [0, \frac{1}{2}]$ is monotone.
	\end{cor}


\begin{thebibliography}{99999999}
		
		\bibitem{alm00} Ll. Alsed\`{a}, J. Llibre and M. Misiurewicz,
		\emph{Combinatorial Dynamics and Entropy in Dimension One},
		Advanced Series in Nonlinear Dynamics (2nd edition) \textbf{5} (2000),
		World Scientific Singapore (2000)
		
		\bibitem{Ba} S. Baldwin, \emph{Generalisation of a theorem of
			Sharkovsky on orbits of continous real valued functions,} Discrete
		Math. \textbf{67} (1987), 111--127.
		
		\bibitem{BB1} S. Bhattacharya, A. Blokh, \emph{Very badly ordered
			cycles of interval maps}, Journal of Difference Equations and
		Applications \textbf{26} (2020), 1067-1084
		
		\bibitem{BB2} S. Bhattacharya, A. Blokh, \emph{Over-rotation intervals of bimodal interval maps},
		Journal of Difference Equations and Applications \textbf{26} (2020), 1085-1113
		
		\bibitem{B1} A. Blokh, \emph{On Rotation Intervals for Interval
			Maps}, Nonlineraity \textbf{7}(1994), 1395--1417.
		
		\bibitem{blo95a} A. Blokh, \emph{The Spectral Decomposition for
			One-Dimensional Maps}, Dynamics Reported \textbf{4} (1995), 1--59.
		
		\bibitem{B2} A. Blokh, \emph{Rotation Numbers, Twists and a
			Sharkovsky-Misiurewicz-type Ordering for Patterns on the Interval},
		Ergodic Theory and Dynamical Systems \textbf{15}(1995), 1--14. 		
		
		\bibitem{BM1} A. Blokh, M. Misiurewicz, \emph{A new order for
			periodic orbits of interval maps}, Ergodic Theory and Dynamical Sys.
		\textbf{17}(1997), 565-574
		
		\bibitem{BS} A. Blokh, K. Snider, \emph{Over-rotation numbers for
			unimodal maps,} Journal of Difference Equations and Aplications
		\textbf{19}(2013), 1108--1132.
		
		\bibitem{MT} J.Milnor and W.Thurston, \emph{On Iterated Maps on
			the Interval}, Lecture Notes in Mathematics, Springer, Berlin
		\textbf{1342}(1988), 465--520.
		
		\bibitem{S} A. N. Sharkovsky, \emph{Coexistence of the cycles of
			a continuous mappimg of the line into itself}, Ukraine Mat. Zh.
		\textbf{16}(1964), 61--71 (Russian).
		
		\bibitem{shatr} A. N. Sharkovsky, \emph{Coexistence of the
			cycles of a continuous mapping of the line into itself}, Internat.
		J. Bifur. Chaos Appl. Sci. Engrg. \textbf{5} (1995), 1263--1273.
		
	\end{thebibliography}
\end{document}